\documentclass{article}
\usepackage{amsmath, amssymb, amsthm, hyperref, enumerate, enumitem}
\usepackage{graphicx, subfigure}
\usepackage{xcolor}

\usepackage{color}

\theoremstyle{plain}
\newtheorem{thm}{Theorem}[section]
\newtheorem{prop}[thm]{Proposition}

\newtheorem{lemma}[thm]{Lemma}

\newtheorem{cor}[thm]{Corollary}

\theoremstyle{definition}
\newtheorem{defn}{Definition}[section]
\newtheorem{rmk}{Remark}[section]

\newcommand{\barMg}{\overline{\mathcal{M}}_g}

\newcommand{\C}{\mathbb{C}}
\newcommand{\R}{\mathbb{R}}

\newcommand{\M}{\mathcal{M}}
\renewcommand{\H}{\mathcal{H}}

\newcounter{bencomments}

\newcounter{jenyacomments}

\title{Coarse density of subsets of $\M_g$}  
\author{Benjamin Dozier, Jenya Sapir}
\date{}
\begin{document}
\maketitle

\begin{abstract}
  Let $\M_g$ be the moduli space of genus $g$ Riemann surfaces. We show that an algebraic subvariety of $\M_g$ is coarsely dense with respect to the Teichm\"uller metric (or Thurston metric) if and only if it is all of $\M_g$.  We apply this to projections of $\operatorname{GL}_2(\R)$-orbit closures in the space of abelian differentials.  Moreover, we determine which strata of abelian differentials have coarsely dense projection to $\M_g$.  
\end{abstract}

\section{Introduction}
\label{sec:introduction}















The moduli space $\M_g$ of closed Riemann surfaces of genus $g$ can be equipped with a natural Finsler metric, the Teichm\"uller metric, which measures the dilatation of the best quasiconformal map between two Riemann surfaces.  In this paper we study some \emph{coarse} metric properties of various natural subsets of $\M_g$.  We say that a subset $S\subset \M_g$ is \emph{coarsely dense} with respect to the Teichm\"uller metric if there is some $K$ such that every point in $\M_g$ is within distance $K$ in the Teichm\"uller metric from some point of $S$.   

\begin{thm}
    \label{thm:main}
  Let $V\subset \M_g$ be an algebraic subvariety.  Then $V$ is coarsely dense in $\M_g$ with respect to the Teichm\"uller metric  iff $V=\M_g$.  
\end{thm}
A similar statement holds if instead of the Teichm\"uller metric, we use the Thurston metric. See page \pageref{sec:ThurstonIntro} for more details.

Our proof uses the compactification $\barMg$ by stable Riemann surfaces.  Intuitively, coarse density of a subset $S\subset \M_g$ should be related to the topological closure $\bar S \subset \barMg$ containing a large portion of the boundary.  It is not true, however, that coarse density of an arbitary $S$ implies that $\bar S$ contains all of $\barMg - \M_g$.\footnote{For instance, take $\bar X$ a stable Riemann surface that is not maximally degenerate, i.e. some piece of $\bar X$ is not a thrice-punctured sphere.  The boundary stratum of $\bar X$ has dimension at least $1$, and so we can take a non-trivial small neighborhood $U$ of $\bar X$ in this boundary stratum.  Then construct $S$ by starting with all stable nodal surfaces (in every boundary stratum), except those in $U$, and smoothing out all the nodes in all possible ways in the plumbing construction.  This $S$ will be coarsely dense, but $\bar X \not\in \bar S$. The point is that small neighborhoods of $\bar X$ in $\barMg$ intersected with $\M_g$ do not contain arbitrarily large balls (though they do have infinite diameter). }
Instead, our proof uses the observation that a coarsely dense subset must contain a sequence of Riemann surfaces with $3g-3$ distinct curves all becoming hyperbolically pinched, but at wildly different rates.  On the other hand, surfaces on a proper subvariety $V$ must satisfy an analytic equation in plumbing coordinates -
this implies a relation among the hyperbolic lengths of pinching curves that prevents them from going to zero at wildly different rates.

\paragraph{Strata and affine invariant manifolds.}

We now describe several corollaries concerning natural subsets arising in Teichm\"uller dynamics that motivated the general theorem above.  

The original motivation came from considering the space of holomorphic Abelian differentials $\Omega \M_g$. 
This is the space consisting of pairs $(X,\omega)$, where $X$ is a genus $g$ Riemann surface, and $\omega$ is a one-form on $X$ that can be written as $f dz$ in local coordinates, where $f$ is a holomorphic function. There is a natural projection $$\pi: \Omega \M_g \to \M_g$$
that takes $(X,\omega)$ to $X$. The theorem allows us to understand something about the coarse geometry of this projection in the following sense.

The space $\Omega \M_g$ naturally decomposes into strata according to the multiplicities of the zeros of $\omega$.  For each partition $\kappa = (\kappa_1,\ldots,\kappa_n)$ of $2g-2$, we get a stratum $\H(\kappa)$ of differentials with $n$ zeros of orders $\kappa_1, \dots, \kappa_n$.


\begin{cor}
    \label{cor:coarse-dense-dim}
  Let $\mathcal{H}(\kappa)$ be a stratum.  Then $\pi(\mathcal{H}(\kappa))$ is coarsely dense in $\M_g$ with respect to the Teichm\"uller metric iff
  \[\dim \mathbb{P}\H(\kappa) \ge 3g-3.\]
\end{cor}

\begin{proof}
  If $\pi(\H(\kappa))=\pi(\mathbb{P}\H(\kappa))$ is coarsely dense, then applying Theorem \ref{thm:main} with $V$ equal to the Zariski closure of  $\pi(\mathbb{P} \H(\kappa))$, we get that $V=\M_g$.  So, using the fact that the dimension of a quasiprojective variety is at least that of the Zariski closure of its image under an algebraic map, we get 
  $$\dim \mathbb{P}\H (\kappa) \ge \dim V = 3g-3.$$
  
  If $\dim\mathbb{P}\H(\kappa) \ge 3g-3$, then by \cite{Gendron18} (Theorem 1.3 or 5.7) or \cite{Chen10} (Proof of Proposition 4.1), $\pi(\H(\kappa))$ contains a Zariski open set of $\M_g$, and in particular is coarsely dense.  (Note that this result is only true for the whole stratum; the connected component $\H^{even}(2,\ldots,2)$ satisfies the dimension condition, but does not dominate a Zariski open set in $\M_g$.) 
\end{proof}

There is a natural action of $\operatorname{GL}_2(\R)$ on each stratum $\H(\kappa)$ given by representing a surface in this stratum as a union of polygons in the plane, and transforming the polygons by the matrix.   The (topological) closures of orbits under this action are of paramount importance for understanding both dynamics on individual surfaces and the structure of the Teichm\"uller metric.  

\begin{cor}
  Let $\M = \overline{\operatorname{GL}_2(\R) (X,\omega)}$ be an orbit closure in a stratum $\H(\kappa)$.  Then $\pi(\mathcal{\M})$ is coarsely dense in $\M_g$ with respect to the Teichm\"uller metric iff
  \[\dim \pi(\M) \ge 3g-3.\]
Here we take dimension of the Zariski closure.  
\end{cor}

\begin{proof}
 By work of Eskin-Mirzakhani-Mohammadi (\cite{emm15}, Theorem 2.1), an orbit closure $\M$ is an affine invariant manifold, and by Filip (\cite{Filip16}, Theorem 1.1) every affine invariant manifold is quasiprojective.  Hence the image $\pi(\M)$ under the algebraic map $\pi$ is a constructible set.  It follows that the Zariski and topological closures of $\pi(\M)$ coincide - call this closure $V$.  Since coarse density is not affected by taking topological closure, $\pi(\M)$ is coarsely dense iff $V$ is coarsely dense.  By Theorem \ref{thm:main}, $V$ is coarsely dense iff $\dim V \ge 3g-3$.  

\end{proof}

In particular, Theorem \ref{thm:main} implies that the projection of a stratum of Abelian differentials, and indeed any affine invariant manifold, cannot be coarsely dense unless it is topologically dense.

%

\paragraph{The Thurston metric.}
\label{sec:ThurstonIntro}
Theorem \ref{thm:main} is also true if we replace the Teichm\"uller metric by the Thurston metric, $d_{Th}$. This is an asymmetric metric on Teichm\"uller space, where the distance from $X$ to $Y$ is defined using the Lipschitz constant for the best Lipschitz map from $X$ to $Y$. Thurston showed that, in fact, 
\[
 d_{Th}(X,Y) = \log \sup_{\alpha} \frac{\ell_Y(\alpha)}{\ell_X(\alpha)}
\]
where the supremum is taken over the set of free homotopy classes of closed curves \cite{Thurston98}.  Here $\ell_X(\alpha)$ denotes the hyperbolic length of $\alpha$ on $X$.  For our purposes, we will take the above formula as the definition.



Because $d_{Th}$ is asymmetric, we can define coarse density with respect to the Thurston metric in two possible ways.  

\begin{defn}
\label{def:Thurston-coarsely-dense}
 For a subset $S\subset \M_g$ the two (generally inequivalent) conditions 
\begin{enumerate}
 \item for each $Y \in \M_g$, there is some $X\in S$ \text{such that} $d_{Th}(X,Y) \leq K$
 \item for each $Y \in \M_g$, there is some $X\in S$ \text{such that} $d_{Th}(Y,X) \leq K$
 \end{enumerate}
give two different definitions of \emph{$K$-coarsely dense}, with respect to the Thurston metric.  
\end{defn}
 
It turns out that our result holds for both definitions:
\begin{thm}
 \label{thm:Thurston}
 Let $V\subset \M_g$ be an algebraic subvariety.  Then $V$ is coarsely dense in $\M_g$ with respect to the Thurston metric (using either definition of coarse density)  iff $V=\M_g$.  
\end{thm}

The proof of this theorem is almost exactly the same as the proof of Theorem \ref{thm:main}, except in the first step. This is explained in Section \ref{sec:ThurstonProof} below.

\paragraph{Acknowledgements.}  We would like to thank Martin M\"oller for suggesting a simplification of our original proof that yielded the more general result proved here. Some of this work was done while the first author was supported by the Fields Institute during the Thematic Program on Teichmüller Theory and its Connections to Geometry, Topology and Dynamics.  Portions of this work were done while the second author was at the thematic program on Geometric Group Theory at MSRI, and while holding a visiting position at the Max Planck Institute for Mathematics in Bonn, Germany.

\section{Plumbing coordinates}
\label{sec:plumbing-coordinates}
The Deligne-Mumford compactification $\barMg$ is obtained by adding to $\M_g$ all genus $g$ nodal Riemann surfaces that are stable, i.e. all components of the complement of the set nodes have negative Euler characteristic.  This space $\barMg$ is compact and has the structure of a complex orbifold, extending the natural complex orbifold structure of $\M_g$.  

The \emph{plumbing construction} gives a system of local analytic coordinates for $\barMg$.  We will only need to understand neighborhoods of a \emph{maximally degenerate} point $\bar X \in \barMg$ (i.e. every component of the complement of the nodes is a thrice-punctured sphere), so we will only describe the coordinates for such points.  Such a point $\bar X$ may have automorphisms which cause $\barMg$ to have an orbifold point at $\bar X$.  To deal with this issue, we assume that we have passed, locally, to a manifold $\mathcal{U}$ that finitely covers the neighborhood.

Coordinates $t_1,\ldots,t_{3g-3}$ on a sufficiently small $\mathcal{U}$ are obtained as follows.  Let $p_1,\ldots,p_{3g-3}\in \bar X$ be the nodes.  The complement $\bar X  -\{p_1,\ldots,p_{3g-3}\}$ is a union of connected Riemann surfaces with a pair of punctures $(a_i,b_i)$ corresponding to each node $p_i$.  Let $U_i,V_i$ be small coordinate neighborhoods of $a_i,b_i$, respectively, and let $u_i,v_i$ be local analytic coordinates at these points.  By rescaling, we can assume that $U_i\supset \{u_i: |u_i|\le 1\}$ and $V_i\supset \{v_i: |v_i|\le 1\}$.   For $(t_1,\ldots,t_{3g-3})$ in a small neighborhood $ \Omega \subset \mathbb{C}^{3g-3}$ of the origin, we define a Riemann surface $X_{t_1,\ldots,t_{3g-3}}$ as follows.  From $\bar X  -\{p_1,\ldots,p_{3g-3}\}$, we remove the sets $D_U=\{0 < |u_i| < |t_i|\}$ and $D_V=\{0 < |v_i| <|t_i|\}$.  We then glue $u_i \in U-D_U$ to $v_i\in V-D_V$ when
$$u_iv_i=t_i.$$
This defines a possibly nodal Riemann surface; if all $t_i\ne 0$, the surface is smooth.  If all $t_i$ are zero, we get $\bar X$.   Hence we get a \emph{plumbing map}
$$\Phi : \Omega \to \mathcal{U}.$$

\begin{figure}[h!]
 \centering 
 \includegraphics{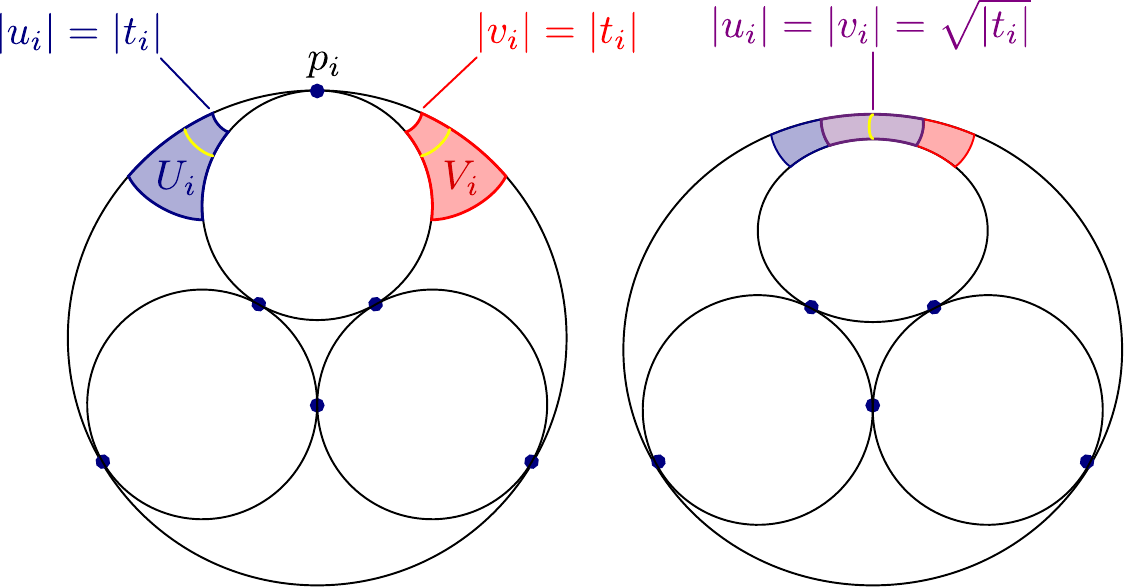}
 \caption{The plumbing construction at the node $p_i$}
 \label{fig:Plumbing}
\end{figure}

\begin{prop}[Plumbing coordinates]
    \label{prop:plumb}
  The plumbing map $\Phi$ is an analytic isomorphism onto an open neighborhood of some lift of $\bar X$ in $\mathcal{U}$.  Here the analytic structure on $\mathcal{U}$ comes from lifting the natural analytic structure on $\barMg$.
\end{prop}

The above result is stated in the Theorem and Corollary of Section 4 in the research announcement \cite{Marden87}.  Proofs of similar results appear in \cite{ Kra90}.  The compactified moduli space $\barMg$ was first studied from the perspective of Teichm\"uller theory by Bers \cite{Bers74} and Abikoff (\cite{Abikoff76} and \cite{Abikoff77}).  These authors constructed $\barMg$ as an analytic space, making heavy use of Kleinian groups.  The plumbing construction seems to have been known quite early; see \cite{Fay73}.  Plumbing parameters as coordinates for $\barMg$ were extensively studied by Earle-Marden \cite{em12} and Kra \cite{Kra90}.

The space $\barMg$ was constructed from an algebraic perspective by Deligne-Mumford \cite{dm69}.
The fact that the analytic structure inherited from the algebraic structure agrees with the analytic structure mentioned in the previous paragraph is due to Hubbard-Koch \cite{hk14}.   Hence algebraic subvarieties are cut out near boundary points of the type discussed above by analytic equations in the plumbing coordinates; we will use this fact in the proof of Theorem \ref{thm:main}.  

\paragraph{Plumbing parameters and hyperbolic lengths.}
We will need to understand the relation between plumbing parameters near a boundary point and the hyperbolic lengths of curves that get pinched at that boundary point.  Conveniently, this is understood precisely (we could get away with much coarser estimates - see Remark \ref{rmk:length}).  

\begin{prop}[\cite{Wolpert10}, p.71]
\label{prop:hyplength}
  For $\bar X$ a maximally degenerate boundary point, let $\alpha_i$ be the pinched curve corresponding to parameter $t_i$.  Then for $X=X_{t_1,\ldots,t_{3g-3}} = \Phi(t_1,\ldots,t_{3g-3})$, the hyperbolic length of $\alpha_i$ satisfies
  \[
  \ell_X(\alpha_i) = \frac{2\pi^2}{\log1/|t_i|} \left( 1 + O\left(\frac{1}{\left( \log 1/|t_i| \right) ^2}\right) \right). 
\]
\end{prop}

\begin{rmk}
  \label{rmk:length}
  In the proof of Theorem \ref{thm:main}, we will not need the full precision of the above proposition, though it is convenient to have.
  To understand the rough order of magnitude, it is sufficient to consider extremal length, since this is comparable to hyperbolic length for short curves.  
  Estimating extremal length in terms of plumbing coordinates is very natural.
  In fact, the highest modulus annulus around a given pinching curve will have an outer boundary that is (approximately) fixed and an inner boundary of radius approximately the plumbing parameter $t$.
  Thus the extremal length is, coarsely, $\frac{2\pi}{\log 1/|t|}$.  


\end{rmk}

\section{Proof of main theorem}
\label{sec:proof-main-theorem}

\subsection{Idea of proof}
The proof consists of the following three steps.
\begin{enumerate}
 \item Given a pants decomposition $\{\alpha_1, \dots, \alpha_{3g-3}\}$, we use coarse density of $V$ to find a sequence $\{X_m\} \subset V$ where 
 \[
  \frac{c_1}{m^i} \leq \ell_{X_m}(\alpha_i) \leq \frac{c_2}{m^i}
 \]
 for constants $c_1$ and $c_2$ that depend only on $K$.
 
 \item The sequence $X_m$ converges to a maximally degenerate point $\bar X \in \bar{M_g}$. Let $t_1, \dots, t_{3g-3}$ be plumbing coordinates for a small neighborhood of $\bar X$. Then by the relationship between hyperbolic lengths of short curves and plumbing coordinates given by Proposition \ref{prop:hyplength}, we get that for any $p \in \mathbb N$ and for all $j > i$,
   \begin{align}
     \label{eq:dominate}
      |t_j(m)| = o(\left| t_i(m)^p \right|)
   \end{align}
 
 where $t_i(m)$ is the $i^{th}$ plumbing coordinate for $X_m$. That is, the $j^{th}$ plumbing coordinate is dominated even by high powers of the previous plumbing coordinates.
 
 \item The fact that plumbing coordinates are analytic and $V$ is not all of $\M_g$ means that they satisfy at least one analytic equation $ f(t_1, \dots, t_{3g-3}) = 0$. To see why this might give a contradiction, consider a specific example where the plumbing coordinates satisfy the polynomial equation 
 \[
 (t_1)^{p_1} + \dots + (t_{3g-3})^{p_{3g-3}} = 0
 \] 
 for some positive exponents $p_1, \dots, p_{3g-3}$. Using (\ref{eq:dominate}), we get that the leading term $t_1^{p_1}$ dominates the others for large $m$, and thus the left hand side cannot possibly equal zero.  
 
 In general, because $f$ is analytic, its Taylor series satisfies strong convergence conditions. Thus, we identify the leading term of the Taylor series, and show that for all $m$ large enough, this leading term dominates the (infinite) sum of the rest of the terms.
 
\end{enumerate}

\subsection{Proof}

\begin{proof}[Proof of Theorem \ref{thm:main}]
  One direction is trivial.

  For the other direction, suppose for the sake of contradiction that $V \ne \M_g$ and that $V$ is $K$-coarsely dense in $\M_g$.
We will use this assumption to construct a sequence of hyperbolic surfaces $X_m \in V$, $m \to \infty$, given by fixing a pants decomposition, and taking the cuffs $\alpha_i$ to have very widely separated lengths, all going to $0$, and show that this contradicts the dimension assumption.  
  
  First, consider the preimage $\tilde V$ of $V$ in Teichmuller space, $\mathcal T_g$. Fix a pants decomposition $\{\alpha_1, \dots, \alpha_{3g-3}\}$ of $S$. Take a sequence of surfaces $Y_m \in \mathcal T_g$ as follows. If $\ell_{Y_m}(\alpha_i)$ denotes the length of $\alpha_i$ with respect to a point $Y_m$, let
  \[
   \ell_{Y_m}(\alpha_i) = \frac{1}{m^i}
  \]
  We specify $Y_m$ by choosing arbitrary twists about each $\alpha_i$.
  
  Since we assume that $V$ is $K$-coarsely dense in $\mathcal M_g$ for some $K > 0$, its lift $\tilde V$ is also $K$-coarsely dense in $\mathcal T_g$. So, we can find some sequence $X_m\in \tilde V$ with $d_{Teich}(X_m, Y_m) \le K$.  A lemma of Wolpert (see \cite[Lemma 12.5]{fm12}) gives that
  \begin{align}
    \label{eq:wolpert}
    \frac 1c \ell_{Y_m}(\alpha_i) \leq \ell_{X_m}(\alpha_i) \leq c \cdot \ell_{Y_m}(\alpha_i)
  \end{align}
  for a constant $c$ depending only on $K$.  So
  \[
   \frac 1c \frac{1}{m^i} \leq \ell_{X_m}(\alpha_i) \leq c \frac{1}{m^i}.
 \]
 for all $m$.
  
  Project the sequence $X_m$ down to $V$. Abusing notation, we will still refer to it as $X_m$. Then in the Deligne-Mumford compactification $\overline{\mathcal{M}}_g$, $\{X_m\}$ converges to the boundary point $\bar X$ where all of the curves $\alpha_i$ are pinched to nodes.
  
  Now by Proposition \ref{prop:plumb}, we can take analytic plumbing coordinates $\mathbf{t} = (t_1, \dots, t_{3g-3}) \in \C^{3g-3}$ for $\mathcal{U}$ a manifold that finitely covers a small neighborhood of $\bar X$ in $\overline{\mathcal M}_g$.  Let $\mathbf{t} (m)= (t_1(m), \dots, t_{3g-3}(m))$ be the plumbing coordinates for (an appropriately chosen lift) of  $X_m$. By Proposition \ref{prop:hyplength}, we have
  \[
   \frac 1c \frac{1}{m^i} \leq \frac{2\pi^2}{\log1/|t_i(m)|} \left( 1 + O\left(\frac{1}{\left( \log 1/|t_i(m)| \right) ^2}\right) \right) \leq c \frac{1}{m^i} 
  \]
  In particular, for $m$ large enough, there is another constant $c'$ so that 
  
   
  \[
   e^{-c'm^i} \leq |t_i(m)| \leq e^{-\frac 1{c'} m^i}.
  \]
  In particular, for all $j > i$ and for any positive integer $p$,
  \begin{align}
    \label{eq:bound}
    |t_j(m)| = o(\left| t_i(m)^p \right|) 
  \end{align}
  as $m \to \infty$.  
  
  These coordinates are analytic, and we are assuming that $V$ is locally cut out by algebraic equations in $\overline{\mathcal M}_g$. This means that in this coordinate chart, (the local lift to the manifold cover $\mathcal{U}$) of $V$ is given as the solution to analytic equations in the plumbing coordinates. Since we are assuming that $V\ne \M_g$, there is at least one such equation:
  \[
   f(t_1, \dots, t_n) = 0,
  \]
  where $f$ is a complex analytic function (not identically equal to zero) defined on a connected neighborhood of the origin.   Here $n=3g-3$.

  Now consider the Taylor series 
  \[
  f(\mathbf{t}) = \sum_\alpha c_\alpha \mathbf{t}^\alpha,
  \]
  where the sum is over all multi-indices $\alpha=(\alpha_1,\ldots,\alpha_n)$ with non-negative integral coefficients (and $\mathbf{t}^\alpha$ means $t_1^{\alpha_1}\cdots t_n^{\alpha_n}$).  Since $f$ is analytic, the sum converges absolutely for small $\mathbf{t}$.

  We wish to identify the term in the sum that is dominant along our sequence $X_m$.  To this end, consider a lexicographic total ordering $\succ$ defined by setting
  \[
  (\alpha_1,\ldots,\alpha_n) \succ (\beta_1,\ldots,\beta_n)
  \]
  iff for the largest $k$ such that $\alpha_k\ne \beta_k$, we have $\alpha_k > \beta_k$ (eg $(1,8,5,2) \succ (2,7,5,2)$ since $7<8$).  Note that any subset of these indices has a smallest element with respect to $\succ$ (i.e. it is a well-ordering).   The  monomials in the Taylor series with smaller $\alpha$ will correspond to \emph{larger} terms for our sequence $X_m$.
  
  Now in the Taylor series, consider the $\alpha$ such that $c_\alpha \ne 0$.   Among these $\alpha$ there is a smallest element $\beta$, and so we can write
  \begin{align}
    f(\mathbf{t}) = c_\beta \mathbf{t}^\beta + \sum_{\alpha \succ \beta} c_\alpha \mathbf{t}^\alpha \label{eq:series}
  \end{align}
  where $c_\beta\ne 0$.

  The strategy is to show that the $\mathbf{t}^\beta$ term dominates the remaining terms along our sequence $X_m$.   By the Cauchy bound (see eg \cite[Proposition 2.1.3]{taylor02}),
  there is some $M$ and $r>0$, such that for any $\alpha$
  \begin{align}
      |c_\alpha| \le \frac{M}{r^{|\alpha|}} \label{eq:cauchy}
  \end{align}
  where $|(\alpha_1,\ldots,\alpha_n)| :=\alpha_1 + \cdots + \alpha_n$.

  Our goal now is to show that $\left| \sum_{\alpha \succ \beta} c_\alpha \mathbf{t}(m)^\alpha\right| = o\left(\left| \mathbf{t}(m)^\beta \right|\right)$ as $m\to\infty$.  We write
  $$\sum_{\alpha \succ \beta} c_\alpha \mathbf{t}^\alpha = S_n(\mathbf{t}) + \cdots + S_1(\mathbf{t}),$$
  where
  \begin{align*}
    S_i(\mathbf{t}) :&= \sum_{\substack{\alpha_n =\beta_n \\ \vdots \\ \alpha_{i+1}=\beta_{i+1} \\ \alpha_i>\beta_i}}  c_\alpha  \mathbf{t}^\alpha 
  \end{align*}
  (Here the sums are taken over all multi-indices $\alpha$ of non-negative integers satisfying the given constraints).  We then get
  \begin{align*}
    |S_i(\mathbf{t})| &\le  \sum_{\substack{\alpha_n \ge\beta_n \\ \vdots \\ \alpha_{i+1}\ge \beta_{i+1} \\ \alpha_i
    > \beta_i}} \left| c_\alpha \mathbf{t}^\alpha \right| \\
             &\le \sum_\alpha\left| c_{\alpha+(0,\dots,0,\beta_i +1, \beta_{i+1},\dots, \beta_n)} \mathbf{t}^{\alpha+(0,\dots,0,\beta_i +1, \beta_{i+1},\dots, \beta_n)} \right| \\
           &= \sum_\alpha \left| c_{\alpha+(0,\dots,0,\beta_i +1, \beta_{i+1},\dots, \beta_n)}\frac{\mathbf{t}^\alpha \mathbf{t}^\beta t_i}{t_1^{\beta_1} \cdots t_{i-1}^{\beta_{i-1} }} \right| \\
    &=\left| \frac{ t_i}{t_1^{\beta_1} \cdots t_{i-1}^{\beta_{i-1} }} \mathbf{t}^\beta \right| \sum_{\alpha} \left| c_{\alpha+(0,\dots,0,\beta_i +1, \beta_{i+1},\dots, \beta_n)}\mathbf{t}^\alpha\right|
  \end{align*}

  Now using the Cauchy bound (\ref{eq:cauchy}), we see that the sum in the last expression above converges to an analytic function for small $\mathbf{t}$ (we also use that the number of multi-indices $\alpha$ with $|\alpha|=k$ grows at most polynomially with $k$).   In particular, the sum is bounded near the origin.
  By (\ref{eq:bound}),
  $$\left|\frac{ t_i(m)}{t_1(m)^{\beta_1} \cdots t_{i-1}(m)^{\beta_{i-1} }} \right|=o(1),$$
  as $m\to\infty$.  Using this and the boundedness of the sum part, we get that $|S_i(\mathbf{t}(m))| = o\left( \left| \mathbf{t}(m)^\beta\right| \right) $ as $m\to\infty$.  Summing, we get
  $$\left|\sum_{\alpha \succ \beta} c_\alpha \mathbf{t}(m)^\alpha \right| \le |S_n(\mathbf{t}(m))| + \cdots + |S_1(\mathbf{t}(m))|=o\left( \left| \mathbf{t}(m)^\beta\right| \right).$$
  Since the points $X_m$ are on the variety $V$, we have $f(\mathbf{t}(m))=0$ for all $m$, so using the expansion (\ref{eq:series}) and rearranging the equation gives
  $$\left|c_\beta \mathbf{t}(m)^\beta \right| = \left|-\sum_{\alpha \succ \beta} c_\alpha \mathbf{t}(m)^\alpha  \right| = o\left( \left| \mathbf{t}(m)^\beta \right|\right),$$
  which gives a contradiction as $m\to\infty$, since $\mathbf{t}(m)^\beta \ne 0$, and by assumption $c_\beta \ne 0$.  
  
 
  
\end{proof}

\section{Proof for the Thurston metric}
\label{sec:ThurstonProof}

We now explain the proof of the main theorem for the Thurston metric (Theorem \ref{thm:Thurston}).
%
%
We can no longer use Wolpert's lemma (\ref{eq:wolpert}) to relate the lengths of curves on pairs of surface that are a bounded distance apart. Instead, the proof hinges on the following lemma:  
  \begin{lemma}
  \label{lem:Thurston-metric-length-bounds}
   Suppose $d_{Th}(X,Y) \leq K$. There is an $\epsilon$ depending only on the genus $g$ so that if $\alpha$ is a simple closed curve with $\ell_Y(\alpha) \leq \epsilon$, then 
   \begin{align}
   \label{eq:claim}
    \frac{1}{c} \cdot  \ell_Y(\alpha) \leq \ell_X(\alpha) \leq c \cdot \ell_Y(\alpha)^{1/c}
   \end{align}
   for a constant $c$ depending only on $K$ and the genus $g$.
\end{lemma}

We will first explain the proof of Theorem \ref{thm:Thurston} given this lemma, and then prove the lemma.

  In contrast to the case of the Teichmuller metric, the ratio $\ell_X(\alpha)/ \ell_Y(\alpha)$ is not bounded in terms of $d_{Th}(X,Y)$. In fact, for any simple closed curve $\alpha$, we can find pairs $X,Y$ where $d_{Th}(X,Y) \leq K$ and $\ell_X(\alpha)/\ell_Y(\alpha)$ is as large as we like.  But the above lemma is enough for our purposes.  

  \begin{proof}[Proof of Theorem \ref{thm:Thurston} given Lemma \ref{lem:Thurston-metric-length-bounds}]
    We will assume that $V$ satisfies the first definition of coarse density from Definition \ref{def:Thurston-coarsely-dense}; Remark \ref{rmk:def} indicates the changes needed if we were to use the second definition instead.
    
    As in the proof of Theorem \ref{thm:main}, we first construct a sequence $Y_m$ of Riemann surfaces converging in $\barMg$ to a maximally degenerate point $\bar X$.
    This time, if $\alpha_1, \dots, \alpha_{3g-3}$ is the shrinking pants decomposition, we need to choose the lengths to be even more widely separated than before, since Lemma \ref{lem:Thurston-metric-length-bounds} gives us less control than Wolpert's lemma (\ref{eq:wolpert}).  Specifically, let 
  \[
   \ell_{Y_m}(\alpha_i) = e^{-m^i}
  \]
  Using our definition of coarse density, for each $m$, there is some $X_m \in V$ with $d_{Th}(X_m,Y_m) \leq K$. So by Lemma \ref{lem:Thurston-metric-length-bounds}, this implies 
  \begin{align}
    \label{eq:len-bds}
      \frac 1c \cdot e^{-m^i} \leq \ell_{X_m}(\alpha_i) \leq c \cdot e^{-m^i/c}
  \end{align}
  for a constant $c$ depending only on $K$ and the genus $g$.

  Again, we take plumbing coordinates $t_1, \ldots, t_{3g-3}$ in a small neighborhood of $\bar X$. Using Proposition \ref{prop:hyplength} in the same way as in the proof of Theorem \ref{thm:main}, we see that
  \[
 e^{\frac{-c'}{ \ell_{X_m}(\alpha_i)}} \leq |t_i(m)| \leq e^{\frac{-1}{c'\cdot \ell_{X_m}(\alpha_i)}}
  \]
  where $c'$ is a constant that depends only on $g$ and $K$.
  Applying the estimate (\ref{eq:len-bds}) to this gives
  \begin{align}
    \label{eq:4}
    e^{-cc'e^{m^i}} \le |t_i(m)| \le e^{-(1/cc')e^{m^i/c}}. 
  \end{align}
  This guarantees that $|t_j(m)| = o(|t_i(m)|^p)$. From here the proof of Theorem \ref{thm:Thurston} proceeds exactly as the proof of Theorem \ref{thm:main} starting from (\ref{eq:bound}).   
\end{proof}

\begin{rmk}
  \label{rmk:def}
    If we were to use the second definition of coarse density, the inequalities in (\ref{eq:len-bds}) would be replaced by
    $(1/c)^c e^{-cm^i} \leq l_{X_m}(\alpha_i) \leq ce^{-m^i}$
    but the structure of the rest of the proof would be the same.
  \end{rmk}
%


\begin{proof}[Proof of Lemma \ref{lem:Thurston-metric-length-bounds}]
The lower bound in the lemma follows directly from the definition of the Thurston metric. In fact, if $d_{Th}(X,Y) \leq K$, then
\[
 \ell_Y(\alpha) \leq e^K \ell_X(\alpha).
\]

To prove the upper bound, we seek a hyperbolic geodesic $\gamma$ (not necessarily simple) on $X$ that crosses $\alpha$ and such that 
\begin{align}
  \label{eq:gamma} 
  \ell_X(\gamma) \le C_1 \log (1/\ell_X(\alpha)) + C_2,
\end{align}
where $C_1,C_2$ are constants that depend only on genus.  The above means that when $\alpha$ is short, the length of $\gamma$ is bounded above by the same order of magnitude as the width of the collar given by the Collar Lemma.  

Choose a Bers pants decomposition $\mathcal{P}$ on $X$, i.e. one for which every cuff curve has length at most the Bers constant $B$ (which depends only on genus). 
If $\alpha$ is not one of the cuff curves in $\mathcal{P}$, then, since it is simple, it must intersect one of the cuff curves, which we take to be $\gamma$.  Since $\ell_X(\gamma)\le B$,  we get (\ref{eq:gamma}) by taking $C_2\ge B$.  
Alternatively, if $\alpha$ is one of the cuffs in $\mathcal{P}$, then look at the pairs of pants $P_1,P_2\in \mathcal{P}$ on either side of $\alpha$.
Take $\gamma$ to be a curve formed by taking ortho-geodesics on $P_1,P_2$, each with both endpoints on $\alpha$, and  concatenating them together along with some piece of $\alpha$.
We can bound the length of each ortho-geodesic from above in terms of $\ell_X(\alpha)$ using the hyperbolic right-angled pentagon formula, which gives $\ell_X(\gamma) \le C_1\log(1/\ell_X(\alpha)) + B'$, where $B'$ depends only on $B$ in this case. 
Combining the two cases gives (\ref{eq:gamma}).

Now, since $d_{Th}(X,Y) \le K$,
\[
\ell_X(\gamma) \ge \frac{1}{e^K} \cdot \ell_Y(\gamma).
\]
By the Collar Lemma, since $\gamma$ must cross the entire collar of $\alpha$, when $\ell_Y(\alpha)$ is sufficiently small, we have
\[
\ell_Y(\gamma) \ge \log (1/\ell_Y(\alpha)).
\]
Applying (\ref{eq:gamma}) and the above two inequalities gives
\[
C_1 \log (1/\ell_X(\alpha)) + C_2\ge  \ell_X(\gamma) \ge \frac{1}{e^K} \cdot \ell_Y(\gamma) \ge \frac{1}{e^K}\log (1/\ell_Y(\alpha)),
\]
which can be rearranged to give the desired second inequality in (\ref{eq:claim}).


%
%
\end{proof}

{\footnotesize
\bibliographystyle{alpha}
  \bibliography{References}
  }

%

\end{document}